\documentclass[conference]{IEEEtran}
\IEEEoverridecommandlockouts
\usepackage{cite}
\usepackage{amsmath,amssymb,amsfonts, comment}
\usepackage{algorithmic}
\usepackage{graphicx}
\usepackage{textcomp}
\usepackage{xcolor}
\usepackage{subcaption} 
\usepackage{mathrsfs}
\usepackage{amsthm}
\def\BibTeX{{\rm B\kern-.05em{\sc i\kern-.025em b}\kern-.08em
    T\kern-.1667em\lower.7ex\hbox{E}\kern-.125emX}}

\newtheorem{theorem}{Theorem}
\newtheorem{corollary}{Corollary}

\begin{document}

\title{Learning with Sparsely Permuted Data: A Robust Bayesian Approach\\
{\footnotesize }
}

\author{\IEEEauthorblockN{1\textsuperscript{st} Abhisek Chakraborty}
\IEEEauthorblockA{\textit{Department of Statistics} \\
\textit{Texas A\&M University}\\
College Station, Texas \\
abhisek\_chakraborty@tamu.edu}
\and
\IEEEauthorblockN{2\textsuperscript{nd} Saptati Datta}
\IEEEauthorblockA{\textit{Department of Statistics} \\
\textit{Texas A\&M University}\\
College Station, Texas \\
saptati@tamu.edu}

}

\maketitle

\begin{abstract}
Data dispersed across multiple files are commonly integrated through probabilistic linkage methods, where even minimal error rates in record matching can significantly contaminate subsequent statistical analyses. In regression problems, we examine scenarios where the identifiers of predictors or responses are subject to an unknown permutation, challenging the assumption of correspondence. Many emerging approaches in the literature focus on sparsely permuted data, where only a small subset of pairs ($k << n$) are affected by the permutation, treating these permuted entries as outliers to restore original correspondence and obtain consistent estimates of regression parameters. In this article, we complement the existing literature by introducing a novel generalized robust Bayesian formulation of the problem. We  develop an efficient posterior sampling scheme by adapting the fractional posterior framework and addressing key computational bottlenecks via careful use of discrete optimal transport and sampling in the space of binary matrices with fixed margins. Further, we establish new posterior contraction results within this framework, providing theoretical guarantees for our approach. The utility of the proposed framework is demonstrated via extensive numerical experiments.

\end{abstract}


\section{Introduction}
Government agencies, national laboratories, and large-scale corporate bodies with access to multiple data sources often integrate datasets, collected or generated at different times and independently of one another, to potentially address extensive research questions. Constraints related to budget, time, and resources prevent such agencies from gathering comprehensive datasets on their own. Record linkage aiming to identify which records across different datasets correspond to the same entity,  is a critical step in this data integration process. Due to the absence of unique entity identifiers amongst multiple data sources, probabilistic record linkage methods are used to measure similarity between quasi-identifiers. Even small error rates in such record matching can contaminate subsequent statistical analyses. Therefore, it is crucial to develop statistical methods to reduce the adverse effects of matching errors.

In regression analysis with response-covariate pairs $\{ (y_i, \mathbf{x}_i) \}_{i = 1}^n$, we  typically assume that each $y_i$ corresponds to the same statistical unit as $\mathbf{x}_i$.  We are interested in the situations where the identifiers of predictors or responses are subject to an unknown permutation, challenging the assumption of correspondence. In particular, when only a small subset of pairs are affected by the permutation, we term it as \emph{sparsely permuted data}. Recovering the permutation between the responses and covariates in such datasets is beneficial if one wishes  to restore the original correspondence, and obtain consistent estimates of the regression parameters. Specifically, we explore the issue in the context of the linear regression problem, where the task is to  regress a study variable $y$ on auxiliary variables $\mathbf{x}$ from a different data source. The number of matching errors, denoted by $k$, is assumed to be a small compared to the total sample size. That is, we assume $k<< n$. Such linear regression taks with sparsely  permuted data routinely arise in various applications, e.g, demographic surveys, pattern recognition in e-commerce transactions, multi-target tracking in radar systems, pose estimation, and correspondence estimation in computer vision, to name a few.

To frame the problem more precisely, assume we have data $(y_1, \mathbf{x}_1), \ldots, (y_n, \mathbf{x}_n)$ derived from matching two distinct data files, $A$ and $B$. Here, for each $i$, $y_i \in \mathbb{R}$ is from file $A$, and $\mathbf{x}_i \in \mathbb{R}^d$ is from file $B$. Given that record linkage is prone to errors, some $\mathbf{x}_i$ values may be incorrectly paired with non-corresponding $y_i$ values. Suppose the number of such mismatches is at most $k << n$. In that case, there exists an unknown permutation $\varphi$ on the set $[n] \neq \{1, \ldots, n\}$ such that $\varphi$ reorders at most $k$ indices. As a result, the pairs $(y_1, \mathbf{x}_{\varphi(1)}), \ldots, (y_n, \mathbf{x}_{\varphi(n)})$ follow the classical linear regression model $
y = \mathbf{x}^T \beta + \epsilon$,
where $\epsilon \sim N(0, \sigma^2)$ and $\mathbf{x} \perp \epsilon$. Let $\Pi$ denote the matrix representation of the permutation $\varphi$, and define the design matrix $X = \begin{pmatrix} \mathbf{x}_1 & \cdots & \mathbf{x}_n \end{pmatrix}^\top$. Consequently, the model can be expressed as
\[
\mathbf{y} = \Pi X \beta + \boldsymbol{\epsilon},
\]
where $\mathbf{y} = (y_1, \ldots, y_n)^\top$ and $\boldsymbol{\epsilon} = (\epsilon_1, \ldots, \epsilon_n)^\top$.
The goal is to accurately infer the parameters $\Pi$ and $(\beta,\sigma^2)$ from the partially mismatched pairs $\{y_i, \mathbf{x}_i\}_{i=1}^n$. The extent of mismatches can be quantified using the Hamming distance between permutation matrix $\Pi$ and the identity matrix $I_n$, defined as 
$$d_H(\Pi, I_n)=|\{i : \Pi_{ii} = 0\}|.$$
In general, we say that a permutation matrix $\Pi$ is  $k$-sparse if $d_H(\Pi, I_n) \le k$. Further,  suppose $\mathcal{P}_n$ denote the set of all permutation matrices in $\mathbb{R}^{n \times n}$, and define a subset $\mathcal{P}_{n,k} \subset \mathcal{P}_n$ as follows:
\[
\mathcal{P}_{n,k} = \{\Pi\in\mathcal{P}_n \mid d_H(\Pi, I_n) \le k\}.
\]
Then, to infer about the parameters of interest via maximum likelihood estimation, one may solve the optimization problem, $(\widehat{\Pi}, \hat{\beta}, \hat{\sigma^2}) = $
$
\mbox{\rm argmax}_{(\Pi, \beta, \sigma) \in \mathcal{P}_{n,k} \times \mathbb{R}^d}\big[ -\frac{1}{2\sigma^2}\|\Pi X \beta - \mathbf{y}\|_2^2 -\frac{n}{2} \log \sigma^2\big].
$
In this article, to complement the existing frequentist literature, we aim to present a generalised robust Bayesian formulation of the problem that enables automatic uncertainty quantification, and develop novel posterior contraction results. Before delving further into details, we will briefly review the key strands of literature relevant to our endeavor.



\subsection{Related works}
In \cite{Neter1965}, assuming the probabilities for matching a response $y_i$ to a covariate record $\mathbf{x}_j$ in the sample, is known or can be estimated, it was shown that the ordinary least squares estimate of $\beta$ is usually biased. The following seminal works \cite{DeGroot1971, DeGroot1976, DeGroot1980} operate under the assumption that $(y_i, \mathbf{x}_i)$  pairs follow a bivariate normal distribution up to a permutation, attempt to recover  the permutation and consistently estimate the correlation coefficient of the bivariate normal distribution. 
\cite{abid2017linear} proposed an estimator of $\beta$ that is shown to be consistent for dimension $d = 1$, and shows promising empirical results even for $d>1$. \cite{Pananjady2016}  established precise conditions on the signal-to-noise ratio, sample size \( n \), and dimension \( d \), determining when  the permutation matrix \( \Pi \) can be exactly or approximately recovered, assuming the entries of the matrix \( X \) are independently drawn from a standard Gaussian distribution.   \cite{collier2016minimax} introduced a rigorous framework to examine the problem of the permutation estimation from a minimax standpoint, and established  optimality  of various natural estimators.

Bayesian approaches for record linkage \cite{steorts2015entity, steorts2016bayesian, sadinle2017bayesian, taylor2023fast} have garnered much attention in recent years too.  \cite{tancredi2011hierarchical} proposed a hierarchical Bayesian approach for linking statistical records observed on different occasions.  \cite{gutman2013bayesian} devised a Bayesian procedure that simultaneously models the record linkage process and the associations between variables in the two files, thereby enhancing matching accuracy and reducing estimation bias. Readers may refer to  \cite{binette2022almost} for a comprehensive review.

An interesting emerging strand  in the frequentist literature operates under the  assumption of the sparsely permuted data, i.e, $k<<n$. Such approaches look at the permuted  entries in the data set as outliers. For instance, \cite{slawski2019} proposed a robust regression formulation of the linear regression problem with sparsely permuted data, and presented a two-stage procedure to first estimate $\beta$, and then utilize this estimate of $\beta$ to recover the permutation $\Pi$. Subsequently, \cite{slawski2020two} extends it to the case with multiple response variables, and again propose a two-stage method  under the assumption that most pairs are correctly matched.  Our goal here is to adopt this view point of handling mismatches as contamination, and develop a generalised Bayesian  framework for simultaneous inference on $(\beta, \Pi)$. The joint robust Bayesian inference  can  circumnavigate the issues with lack of propagation of uncertainties between the two stages of the existing  procedures, and enable us to conduct  theoretical analysis of joint posterior of $(\beta, \Pi)$.

To that end, we appeal to  the rich literature on robust Bayesian procedures,
that provide reliable probabilistic inference even under mild model misspecification and/or presence of outliers. One strategy to develop robust Bayesian linear regression consists of assuming heavy tailed parametric error distributions to construct the likelihood (\cite{sutradhar2003robust, lee2002robust}), which in combination with a suitable prior specification, yields a robust posterior. Notably, non-parametric Bayes methods are also routinely used to guard against model misspecification \cite{muller2004nonparametric, 10.1214/009053606000000029, 10.2307/24310519}.  However, the issues regarding the presence of a large number of uninterpretable parameters in such non-parametric Bayes procedures has led to a recent proliferation of alternative strands in the literature, based on pseudo-likelihood \cite{RePEc:eee:econom:v:115:y:2003:i:2:p:293-346, 2008,JMLR:v18:16-655, safebayes,  doi:10.1080/01621459.2018.1469995} and empirical likelihood  \cite{10.2307/30042042, 10.1093/biomet/92.1.31,chib2021bayesian, chakraborty2023robustprobabilisticinferenceconstrained, e26030249}.
In particular, \cite{doi:10.1080/01621459.2018.1469995} introduced a coherent pseudo likelihood based  approach to carry out robust Bayesian inference under mild model perturbations. Instead of conditioning directly on the observed data to define the posterior of the parameters of interest, this method conditions on a neighborhood of the empirical distribution of the observed data. When such neighborhood is defined based on the relative entropy, the resulting coarsened posterior can simply be approximated by tempering the likelihood, i.e, raising it to a fractional power. This allows for inference to be implemented easily with standard methods, and analytical solutions can be obtained when using conjugate priors.

\subsection{Our contributions}
In this article, as indicated earlier, we treat the permuted data entries in the linear regression task with sparsely permuted data as outliers. Complementing the existing literature, we present a novel generalized robust Bayesian formulation of the problem. Further, adapting the fractional posterior framework, we develop an efficient posterior sampling scheme, carefully navigating the key computational bottlenecks. Further, under  the proposed fractional posterior framework, we establish novel posterior contraction results to provide theoretical guarantees. Finally, the versatility of the proposed solution is illustrated by its straightforward extension to the quantile regression problem with sparsely permuted data.

\section{Proposed methodology}
Let $[t]$ denote the set of integers $\{1, 2,\ldots, t\}$. The  Hamming distance between two permutation matrices $\Pi_1=((\pi^{(1)}_{ij}))_{i=1, j=1}^n$ and $\Pi_2=((\pi^{(2)}_{ij}))_{i=1, j=1}^n$ is defined as $d_H(\Pi_1, \Pi_2)=\sum_{i=1}^n\sum_{j=1}^n |\pi^{(1)}_{ij}-\pi^{(2)}_{ij}|$. Let $\mathcal{P}_n$ denote the set of all $n\times n$ permutation matrices, and define the subsets $\mathcal{P}_{n,a} = \{\Pi\in\mathcal{P}_n \mid d_H(\Pi, I_n) \le a\}$ for any $a\in[n]$, where $I_n$ is the diagonal matrix of order $n$. Let $\mathcal{U}(\mathcal{P}_{n,a})$ denote the uniform distribution over the restricted space of permutation matrices $\mathcal{P}_{n,a}$.

\subsection{Model and prior specification}\label{ssec:main_model}
Suppose  we observe data $(y_i, \mathbf{x}_i)\in\mathbf{R}\times\mathbf{R}^d, i\in[n]$ with mismatches in at most $k (<< n)$ entries. We consider the model
\begin{align}\label{eqn:model1}
 \mathbf{y} = \Pi X \beta + \boldsymbol{\epsilon},\quad \boldsymbol{\epsilon} \sim\mbox{N}_n(\mathbf{0}, \sigma^2 \mbox{I}_n),  
\end{align}
where  the permutation matrix $\Pi=((\pi_{ij}))_{i=1, j=1}^n$ is assumed to be such that 
\begin{align}\label{eqn:model2}
  d_H(\Pi, I_n)=\sum_{i=1, j=1}^n |\pi_{ij}-1|\leq k.  
\end{align}
Our goal is to develop a fully Bayesian framework to infer the parameters $\Pi$ and $(\beta,\sigma^2)$. To that end, we first specify a uniform prior on $\Pi$ as follows
\begin{align}\label{eqn:model3}
\Pi\sim\mathcal{U}(\mathcal{P}_{n,a}).
\end{align}
That is, we assume $\pi(\Pi) =\frac{1}{|\mathcal{P}_{n,a}|}$, for $\Pi\in\mathcal{P}_{n,a}$. One may specify more elaborate priors that still impose  the restriction $d_H(\Pi, I_n)\leq k$, but we choose to keep our prior specification simple to ensure brevity of exposition. The hierarchical specification is completed by positing non-informative priors on the regression coefficients $\beta$ and error variance $\sigma^2$ as follows
\begin{align}\label{eqn:model4}
&\beta\sim \mbox{N}_d(0, 1000\mbox{I}_d),\quad \sigma^2\sim \mbox{TN}_d(0, 1000; 0, \infty),
\end{align}
where $\mbox{TN}_d(a, b; c, d)$ denotes the normal distribution with mean $a$ and variance $b$, restricted to the interval $(c,d)$. 
Under the model and prior specification in \eqref{eqn:model1}-\eqref{eqn:model4}, the joint posterior of $(\beta, \sigma^2, \Pi)$ given data can be expressed as 
\begin{align*}
&\pi(\beta, \sigma^2, \Pi\mid\{\mathbf{x}_i, y_i\}_{i=1}^n)\ \propto\notag\\
&\bigg\{\frac{1}{\sigma^2}\exp\bigg(-\frac{||\mathbf{y} - \Pi X\beta||^2}{\sigma^2}\bigg)\bigg\}\pi(\beta)\pi(\sigma^2)\pi(\Pi).
\end{align*}

One may simply develop a Gibbs sampling scheme to sample from the joint posterior above to carry out a fully Bayesian inference, via cyclically sampling from the full conditional distribution of each of the parameters given others. However, sampling from the full conditional distribution of $[\Pi\mid\beta, \sigma^2]$ pose a difficult combinatorial problem, and may require a judicious initialization scheme to ensure sampling efficiency. Instead, we propose to treat the permuted data entries as outliers, and develop robust Bayesian approach to carry out the inference. The proposed approach, in combination with a carefully crafted computational scheme, produce automatic uncertainty quantification. '

The key idea to develop the robust Bayesian procedure \cite{doi:10.1080/01621459.2018.1469995} is as follows. Customarily, one conditions the event that the observed data is generated from the true data generating mechanism, to define the posterior. Instead, one can condition on the event that the observed data lies in a neighbourhood of the true data generating mechanism to define a robustified posterior, that yields reliable inference by  allowing for mild perturbations in the data generating mechanism by construction. We clarify the details in the sequel.

Let  $\{\mathbf{x}^{\star}_{i}, y^{\star}_i\}_{i=1}^n$ denote an unobserved  data set without any permutation, identically and independently generated from the linear regression model. However, the observed sparsely permuted data $\{\mathbf{x}_{i}, y_i\}_{i=1}^n$  are actually a mildly corrupted version of $\{\mathbf{x}^{\star}_{i}, y^{\star}_i\}_{i=1}^n$ in the sense that 
\begin{align*}
    \mbox{\rm D}(\hat{\rm P}_{\{{\mathbf{x}^{\star}_{i}, y^{\star}_i\}_{i=1}^n}}, \hat{\rm P}_{\{{\mathbf{x}_{i}, y_i\}_{i=1}^n}}) < r,
\end{align*}
for some statistical distance \(\mbox{D}\) and some \(r > 0\), where \(\hat{\rm P}_{\{{\mathbf{x}_{i}, y_i\}_{i=1}^n}} = \frac{1}{n} \sum_{i=1}^{n} \delta_{\mathbf{x}_i, y_i}\) denotes the empirical distribution of $\{\mathbf{x}_{i}, y_i\}_{i=1}^n$. If there were no permutation, then we should use the standard posterior, conditioning  on the event that $\{\mathbf{x}^{\star}_{i}, y^{\star}_i\}_{i=1}^n =  \{\mathbf{x}_{i}, y_i\}_{i=1}^n$. However, due to the corruption arising from the permutation, we  condition on the event that $\mbox{\rm D}(\hat{\rm P}_{\{{\mathbf{x}^{\star}_{i}, y^{\star}_i\}_{i=1}^n}}, \hat{\rm P}_{\{{\mathbf{x}_{i}, y_i\}_{i=1}^n}}) < r$. In other words, rather than the standard posterior \(\pi(\beta,\sigma^2, \Pi \mid \{\mathbf{x}_{i}, y_i\}_{i=1}^n )\), we should consider the modified posterior
\begin{align*}
    \pi_{\rm mod}(\beta,\sigma^2, \Pi \mid \mbox{\rm D}(\hat{\rm P}_{\{{\mathbf{x}^{\star}_{i}, y^{\star}_i\}_{i=1}^n}}, \hat{\rm P}_{\{{\mathbf{x}_{i}, y_i\}_{i=1}^n}}) < r ).
\end{align*}
Further,  we assume an exponential prior on $r$ with mean $ 1/\kappa$, independently of $(\beta, \sigma^2, \Pi)$ and data. Then, one can show that 
\begin{align*}
    &\pi(\beta,\sigma^2, \Pi \mid \mbox{\rm D}(\hat{\rm P}_{\{{\mathbf{x}^{\star}_{i}, y^{\star}_i\}_{i=1}^n}}, \hat{\rm P}_{\{{\mathbf{x}_{i}, y_i\}_{i=1}^n}}) < r )\notag\\
    \approx&\exp\bigg\{-\kappa \mbox{\rm D}(\hat{\rm P}_{\{{\mathbf{x}^{\star}_{i}, y^{\star}_i\}_{i=1}^n}}, \hat{\rm P}_{\{{\mathbf{x}_{i}, y_i\}_{i=1}^n}})\bigg\}\pi(\beta)\pi(\sigma^2)\pi(\Pi).
\end{align*}
Readers can refer to \cite{doi:10.1080/01621459.2018.1469995} for details of the derivation of such \emph{coarsened posteriors}in the general set up. Further, suppose  the statistical discrepancy $\mbox{D}$ to be the relative entropy between two probability measures, defined by $D(p_0, p_1) =\int p_0\log\big(\frac{p_0}{p_1}\big)dp_0.$ Then, the modified posterior further simplifies to 
\begin{align}\label{eqn:fracpost}
     &\pi_{\rm mod}(\beta,\sigma^2, \Pi \mid \mbox{\rm D}(\hat{\rm P}_{\{{\mathbf{x}^{\star}_{i}, y^{\star}_i\}_{i=1}^n}}, \hat{\rm P}_{\{{\mathbf{x}_{i}, y_i\}_{i=1}^n}}) < r )\stackrel{\text{approx}}{\propto}\notag\\
    &\bigg\{\frac{1}{\sigma^2}\exp\bigg(-\frac{||\mathbf{y} - \Pi X\beta||^2}{\sigma^2}\bigg)\bigg\}^{\alpha}\pi(\beta)\pi(\sigma^2)\pi(\Pi),
\end{align}
for some $\alpha\in(0, 1)$ that depends on $\kappa$. In what follows, $\alpha$ is a hyper-parameter that needs to be specified.
In numerical studies, we shall demonstrate that $\alpha\approx 1/n$ provides reasonable results across varied data generating mechanisms.

Before we present the computational details, we shall briefly discuss the rationale behind adopting the fractional posterior approach, compared to other existing robust Bayesian approaches. The advantages are two fold. First, the fractional posterior approach allows us to adapt existing posterior computation schemes for efficient sampling from the standard posterior, facilitating their application to our specific case.
Secondly, the fractional posterior approach lends itself to concise and rigorous theoretical analysis. We have developed novel posterior contraction results for inference arising from the modified joint posterior in \eqref{eqn:fracpost}, providing key insights into the proposed methodology. Similar results could potentially be derived for alternative pseudo-posterior approaches or methods that model the error \(\mathbf{\varepsilon}\) using heavy-tailed distributions, and this presents a promising direction for future research.

\subsection{Posterior computation}\label{ssec:mcmc}
We present a Gibbs sampling scheme to sample from the joint posterior in \eqref{eqn:fracpost} to carry out a fully Bayesian inference, via cyclically sampling from the full conditional distributions of $[\beta, \sigma^2\mid\Pi]$ and $[\Pi\mid\beta,\sigma^2]$. 

{\textbf{Step 1.}} To sample from the full conditional distribution of $[\beta, \sigma^2\mid\Pi]$, one may develop a naive Metropolis--Hastings or Gibbs sampling scheme. However, depending on the choice of priors on $\beta$ and $\sigma^2$, one needs to tailor the sampler. To ensure ease of the practitioners, we utilize an off-the-shelf Hamitonian Monte Carlo (HMC) algorithm, available in the probabilistic programming language Stan \cite{rstan}.

Without getting into minute details, we provide a quick primer on key ideas of the HMC algorithm (\cite{neal2011mcmc, betancourt2017conceptual}). HMC is a powerful MCMC method for sampling from posterior distributions, leveraging concepts from physics to improve efficiency. A target distribution, say $\pi(\eta)$, is explored by simulating Hamiltonian dynamics, which combines the potential energy (negative log-posterior) and kinetic energy (typically Gaussian). Starting with an initial \(\eta\) and momentum \(r\), HMC uses leapfrog steps for numerical integration: update momentum (half-step) $r_{t + \frac{1}{2}} = r_t - \frac{\epsilon}{2} \nabla_\eta \log{\pi}(\eta_t),$ update position (full-step) $\eta_{t+1} = \eta_t + \epsilon r_{t + \frac{1}{2}},$ and update momentum (half-step) $r_{t+1} = r_{t + \frac{1}{2}} - \frac{\epsilon}{2} \nabla_\eta \log\pi(\eta_{t+1}).$ These steps are carried out to propose new states, which are then accepted or rejected via a Metropolis-Hastings correction, ensuring detailed balance. Thus, HMC efficiently traverses the parameter space, making it suitable for semi-automated Bayesian inference. 

\textbf{Step 2.} Next, we turn our attention to sampling from the full conditional distribution of $[\Pi\mid\beta, \sigma^2]$. This pose a difficult combinatorial problem. To that end,  we propose  a carefully crafted computational scheme, utilizing  a discrete optimal transport (\cite{Villani2003TopicsIO, cuturi2013sinkhorn}) guided proposal scheme, and followed by sampling scheme in space of binary matrices with fixed margins (\cite{miller2013exact, BRUALDI198033, curveball, wang2020fast}).

We first develop a computationally convenient MC-EM algorithm, where instead of sampling from the full conditional distribution of  $[\Pi\mid\beta, \sigma^2]$, we update the chain with the posterior mode of  $[\Pi\mid\beta, \sigma^2]$. Specifically,  to compute the posterior mode, we go over the following steps.
\begin{itemize}
\item \textbf{Cost matrix.} First, given $\beta$ and $\sigma^2$, we compute the $n\times n$  cost matrix 
{\small
\begin{align}\label{eqn:ot_cost}
    \mbox{L}= ((l_{ij})) = \bigg(\bigg(\alpha\bigg[\frac{(y_i -\beta^T(\Pi X)_{j})^2}{2\sigma^2} + \log\sigma^2\bigg]\bigg)\bigg),
\end{align}}
 where $\mathbf{1}_{n}$ is a vector of $n$ $1$s, and $(\Pi X)_{j}$ denotes the $j$-th row of $\Pi X$. 
\item  \textbf{Discrete optimal transport.} Next, we define the polytope of $n\times n$ binary  matrices 
$$\mathcal{P}^{n\times n}:=\mbox{U}(\mathbf{1}_n, \mathbf{1}_n):=\{\mbox{B}\mid \mbox{B} 1_{n} =  1_{n};\ \mbox{B}^{T}1_{n} =  1_{n}\},$$
and solve the constrained binary optimal transport problem 
\begin{align}\label{eqn:opt_B}
    \mbox{B}_{\rm opt}= \mbox{\rm argmin}_{\mbox{B}\in \mbox{U}( 1_{n},\  1_{n})} \langle\mbox{B}, \mbox{L} \rangle,
\end{align}
where $\langle\mbox{B}, \mbox{L}\rangle = \rm{tr}(\mbox{B}^T \mbox{L})$. 
\end{itemize}
This  describes the scheme to compute the mode of the full conditional distribution of $[\Pi\mid\beta,\sigma^2]$, and completes the MC-EM algorithm to maximize the joint posterior of $[\beta, \sigma^2, \Pi]$ in \eqref{eqn:fracpost}.  

One may replace the last optimization step  by a non-uniform sampling step of binary matrices with fixed margins, to describe a complete Gibbs sample. We describe this in the sequel. We need to sample from the space of binary matrices $\mbox{U}(\mathbf{1}_n, \mathbf{1}_n)$ according to the non-uniform probability distribution defined by the weight matrix 
$$\Omega =((\omega_{ij})):= \exp{(-L)}:=((\exp(-l_{ij})),$$
where $\mbox{L}$ is as in \eqref{eqn:ot_cost}. Note that, the likelihood associated with a permutation matrix $H\in\mbox{U}(\mathbf{1}_n, \mathbf{1}_n)$ is
$$\mbox{P}(H) = (1/\zeta) \prod_{i,j} \omega_{ij}^{h_{ij}},\quad \zeta = \sum_{H\in\mbox{U}(\mathbf{1}_n, \mathbf{1}_n)}\ \prod_{i,j} \omega_{ij}^{h_{ij}}.$$ Let  $\mbox{U}^{\prime}(\mathbf{1}_n, \mathbf{1}_n) = \{H\in\mbox{U}(\mathbf{r}, \mathbf{c}):\ P(H)>0\}$ denote the subset of matrices in $\mbox{U}(\mathbf{1}_n, \mathbf{1}_n)$ with positive probability. Then, for  $H_1,H_2\in\mbox{U}^{\prime}(\mathbf{1}_n, \mathbf{1}_n)$,  the relative probability of the two observed matrices is
$$\frac{\mbox{P}(H_1)}{\mbox{P}(H_2)} = \frac{\prod_{\{i,j: h_{1, ij}=1, h_{2, ij}=0\}} \omega_{ij}^{h_{1, ij}}}{  \prod_{\{i,j: h_{1, ij}=0, h_{2, ij}=1\}} \   \omega_{ij}^{h_{2, ij}}}.$$ Further, we note that the matrices
\[
\begin{pmatrix}
1 & 0 \\
0 & 1
\end{pmatrix}
\quad \text{and} \quad
\begin{pmatrix}
0 & 1 \\
1 & 0
\end{pmatrix}
\]
are referred to as checker-board matrices.  With these notations, we adapt \cite{wang2020fast, e26010063, chakraborty2023fairclusteringhierarchicalfairdirichlet}  and proceed as follows.
\begin{itemize}
    \item \textbf{Initialization.} At iteration $t=0$, we set the initial permutation matrix $A_0 $ to $ \mbox{B}_{\rm opt}$, as computed in \eqref{eqn:opt_B}. Let the total number of iterations be $T$. 
    
    \item \textbf{Proposal.} Increase $t\to t+1$. 
    \begin{itemize}
        \item Choose one row and one column $(r_1,c_1)$ uniformly at random.

        \item  \emph{If} $A_{t-1}(r_1,c_1) = 1$, choose a column $c_2$ at random among all the $0$ entries in $r_1$, and a row $r_2$ at random among all the $1$ entries in $c_2$. \emph{Else} choose a row $r_2$ at random among all the 1 entries in $c_1$, and a column $c_2$ at random among all the $0$ entries in $r_2$. 

        \item \emph{If} the sub-matrix extracted from $r_1, r_2, c_1, c_2$ is a checkerboard unit -- obtain $B_{t}$ from $A_{t-1}$ by swapping the checker-board.
    \end{itemize}
    
    \item\textbf{Acceptance/rejection.}
    \begin{itemize}
        \item Calculate $p_{t}=\mbox{P}(B_{t})/[\mbox{P}(B_{t}) + \mbox{P}(A_{t-1})]$.

        \item Draw $r_t\sim \mbox{Bernoulli}(p_{t})$, and  \emph{If} $r_t = 1$, then set $A_{t}=B_{t}$; \emph{else} set $A_{t}=A_{t-1}$.

        \item \emph{If} the sub-matrix extracted from $r_1, r_2, c_1, c_2$ is \emph{not} a checkerboard unit, set $A_{t}=A_{t-1}$. 
    \end{itemize}
 \end{itemize}
 By construction, the Markov chain above  converges to the correct stationary distribution, as the number of iterations $T\to\infty$. This  describes the scheme to sample from the full conditional distribution of $[\Pi\mid\beta,\sigma^2]$, and completes the MCMC algorithm to sample from the joint posterior of $[\beta, \sigma^2, \Pi]$ in \eqref{eqn:fracpost}. 

\subsection{Alternative problem formulation}\label{ssec:alternative}
One may consider the following alternative formulation  the problem of linear regression with sparsely permuted data \cite{slawski2019},
\begin{align}\label{eqn:alt_mod1}
&\mathbf{y} = \Pi X\beta +\varepsilon\quad \equiv\quad \mathbf{y} = X\beta + \mathbf{f} + \varepsilon,
\end{align}
where $\mathbf{f}=(f_1,\ldots, f_n)^T=(\Pi -I)X\beta$ is a such that 
\[
f_i =
\begin{cases}
= 0 & \text{for no permutation}, \\
\neq 0 & \text{for permutation},
\end{cases}
\]
for $i\in[n]$. Note that, $\mathbf{f}$  is a $n\times1$ sparse vector with $k(<<n)$ non-zero elements. For ease of inference, we may consider a relaxation  appealing to the rich literature on sparse Bayesian learning based on spike-slab priors (\cite{mitchell1988bayesian, george1993variable, ishwaran2005spike, narisetty2014bayesian}), and formulate the problem as follows
\begin{align}\label{eqn:alt_mod2}
&\mathbf{y} = X\beta + \mathbf{f} + \varepsilon,\notag\\
& f_i\overset{\text{i.i.d}}{\sim}\ \frac{k}{n}\mbox{N}(0,\ \sigma^2_{f}) + \bigg(1 - \frac{k}{n}\bigg) \delta_0, \ i\in[n],
\end{align}
where $\delta_0$ is the Dirac delta measure at $0$. One may alternatively consider a continuous shrinkage prior on $\mathbf{f}$ (\cite{park2008bayesian,polson2014bayesian,carvalho2010horseshoe, griffin2010inference, armagan2013generalized, bhattacharya2015dirichlet}) to carry out inference, but we stick to spike slab priors to ensure brevity of presentation.

We note that the formulation in \eqref{eqn:alt_mod2} suffer from two key disadvantages, compared to the model and prior specification in Section \ref{ssec:main_model}. First, the formulation in \eqref{eqn:alt_mod2} is approximate, since we do not take into account the exact definition of $f$ and merely pose it as a $k$-sparse vector. Secondly,  the estimated $\hat{f}$ would only yield the indices of non correspondence, but we cannot recover the estimated permutation matrix $\hat{\Pi}$ from $\hat{f}$. Consequently, the practical appeal of the formulation in \eqref{eqn:alt_mod2}, compared \eqref{eqn:alt_mod1} is reduced, and we do not pursue it further in our numerical studies. However, we theoretically study it in the sequel and establish novel posterior contraction results.  


\section{Theoretical guarantees}
In Bayesian inference, determining posterior contraction rates is crucial to measure how rapidly the posterior distribution converges within a very small neighborhood of the true model as the sample size diverges to infinity. The seminal papers (\cite{Ghoshal2000, Ghoshal2007}) extensively studied the posterior contraction phenomenon, assuming the observations are  i.i.d, under both well-specified and mis-specified setup.  \cite{GhoshnVart2007} presented posterior contraction results in the case of non-i.i.d. observations. More recently, \cite{Bhattacharya2019} introduced conditions for  posterior contraction  in fractional posteriors,  encompassing both well-specified and mis-specified models. In particular, \cite{Bhattacharya2019} demonstrated that the contraction rate of the fractional posteriors solely depend  on the prior mass concentrated within a specific Kullback-Liebler (K--L) neighborhood of the true data generating mechanism. This result is indeed crucial for deriving the contraction rates for fractional posteriors under a wide range of prior specification. The primary advantage of utilizing the conditions for  contraction of the fractional posteriors \cite{Bhattacharya2019}, over the conditions delineated in \cite{GhoshnVart2007} for standard posteriors, is that the latter requires  careful construction of sieves within the parameter space, and the prior mass condition alone is not sufficient to guarantee posterior contraction.  Readers are referred to \cite{Ghoshal2000, Ghoshal2007} for further details on sieve construction and general strategies for deriving posterior contraction results for standard Bayesian posterior. In this section, for the linear regression model with sparsely permuted data, we work with the fractional posterior introduced in \eqref{eqn:fracpost} and   derive novel posterior contraction results. 

To that end, few notations are in order. For a sequence of observations \(W^{(n)}\), suppose we posit a model  \(\mathbb{P}^{(n)}_\theta\) that admits the density \(p_\theta^{(n)}\), and $\theta \in \Theta$ is the parameter of interest.
Let  \(\Pi_n\) be a prior distribution on \(\theta \in \Theta\), and $L_{n, \alpha}(\theta) = \left[ p_\theta^{(n)}(W^{(n)}) \right]^\alpha$ be the fractional likelihood of order \(\alpha \in (0, 1)\),
where \(p_\theta^{(n)}(W^{(n)})\) is the standard likelihood.
For any measurable set $B$ belonging to a $\sigma$-field $\mathcal{B}$, the fractional posterior distribution \(\Pi_{n, \alpha}(\cdot)\) is then defined as
\begin{align*}
\Pi_{n, \alpha}(B \mid W^{(n)}) 
= \frac{\int_{B} L_{n, \alpha}(\theta) \, \Pi_n(d\theta)}{\int_{\Theta} L_{n, \alpha}(\theta) \, \Pi_n(d\theta)}.
\end{align*}
Further, suppose $\theta_0$ is the true parameter value, and $D\big( p_{\theta_0}^{(n)}, p_{\theta}^{(n)}\big)$ is the K--L divergence between $p_{\theta_0}^{(n)}$ and $p_{\theta}^{(n)}$. Let \(\theta^*\) be the parameter value that minimizes $D\big( p_{\theta_0}^{(n)}, p_{\theta}^{(n)}\big)$, that is,
\begin{align}\label{eq:KL_min}
\theta^\ast :\,=\arg\min_{\theta\in\Theta} D\big( p_{\theta_0}^{(n)}, p_{\theta}^{(n)}\big).
\end{align}
 In the well specified cases, that is when $\theta_0\in\Theta$, we have $\theta_0 = \theta^*$.
Finally, we introduce the \(\alpha\)-divergence between \(p^{(n)}_{\theta}\) and \(p^{(n)}_{\theta^\ast}\) with respect to \(p^{(n)}_{\theta_0}\), given by
$$D^{(n)}_{\theta_0, \alpha}(\theta, \theta^\ast) := \frac{1}{\alpha - 1} \log A^{(n)}_{\theta_0,\alpha}(\theta, \theta^\ast),$$
where
$$A^{(n)}_{\theta_0,\alpha}(\theta, \theta^\ast) := \int \left(\frac{p^{(n)}_{\theta}}{p^{(n)}_{\theta^\ast}}\right)^\alpha p^{(n)}_{\theta_0} \, d\mu^{(n)}.$$
 Then, for any given \(\theta'\), we define a K--L neighborhood  with radius \(\varepsilon\) centered around \(\theta'\), denoted by $N_{n}(\theta', \varepsilon; \theta_0)=$
{\footnotesize
\begin{align*}
\left\{\theta \in \Theta : \mbox{E}_{p_{\theta_0}^{(n)}}\bigg[\log \frac{p_{\theta'}^{(n)}}{p_{\theta}^{(n)}})\bigg]\leq n \varepsilon^2, \right.
 \left. \mbox{E}_{p_{\theta_0}^{(n)}}\bigg[\log^2 \frac{p_{\theta'}^{(n)}}{p_{\theta}^{(n)}}\bigg]\leq n \varepsilon^2 \right\},
\end{align*}}
where, 
$\mbox{E}_{p_{\theta_0}^{(n)}}\bigg[\log \frac{p_{\theta'}^{(n)}}{p_{\theta}^{(n)}}\bigg] = \int p_{\theta_0}^{(n)} \log \left(\frac{p_{\theta'}^{(n)}}{p_{\theta}^{(n)}}\right) d\mu^{(n)},$ and 
$\mbox{E}_{p_{\theta_0}^{(n)}}\bigg[\log^2 \frac{p_{\theta'}^{(n)}}{p_{\theta}^{(n)}}\bigg] = \int p_{\theta_0}^{(n)} \log^2 \left(\frac{p_{\theta'}^{(n)}}{p_{\theta}^{(n)}}\right) d\mu^{(n)}$. With the above notations, \cite{Bhattacharya2019} presented  the \emph{prior mass condition} to establish the posterior contraction rate of a fractional posterior.

\begin{theorem}[\cite{Bhattacharya2019}] \label{Thm:contraction}
Fix $\alpha \in (0,1)$, and consider $\theta^*$ as in \eqref{eq:KL_min}. Assume that, the sequence \(\varepsilon_n\) satisfies \(n\, \varepsilon_n^2 \geq 2\), and the prior mass condition 
\begin{align}\label{eq:prior_Mass}
\Pi_n\big(N_{n}(\theta', \varepsilon; \theta_0)\big) \geq e^{-n\,\varepsilon_n^2}
\end{align}
holds.
Then, for any \(D \geq 2\) and \(t > 0\),
\begin{align*}
\Pi_{n,\alpha} \left(\frac{1}{n}D^{(n)}_{\theta_0,\alpha}(\theta,\,\theta^\ast) \geq \frac{D + 3t}{1-\alpha}\, \varepsilon_n^2 \ \Bigg| \, W^{(n)}\right) \leq e^{-t\, n\, \varepsilon_n^2}
\end{align*}
holds with \(\mathbb{P}^{(n)}_{\theta_0}\) probability at least \(1 - 2/(D-1+t)^2 n \varepsilon_n^2\).
\end{theorem}
With the necessary preliminary background in place, we now proceed to derive the main results within our proposed framework.

\subsection{Posterior contraction in the proposed formulation in \eqref{ssec:main_model}}
We shall denote  \(\theta = (\Pi, \beta)\), and  \(\theta_0 = (\Pi_0, \beta_0)\) represents the true parameter values. Since, we are operating under a well-specified set up, we have \(\theta^* = \theta_0\).  Let \(W_i = (y_i, \mathbf{x}_i)\) denote the \(i\)-th data unit, and \(W^{(n)} = (W_1, W_2, \ldots, W_n)\) represent the entire dataset. Utilizing this notation, we proceed to state and prove the prior mass condition in Theorem \ref{prior_mass_cond}, which is essential for establishing the posterior contraction rate of the proposed method in \eqref{ssec:main_model}.

\begin{theorem}\label{prior_mass_cond}
Assume $\sigma^2 = 1$. For  $\alpha \in (0,1)$, define the neighbourhood 
{\small
$$\mathscr{A} = \left\{ \theta = (\Pi, \beta) : \|\Pi - \Pi_0\|_F^2 \leq \delta_\pi, \|\beta - \beta_0\|_2^2 \leq \delta_\beta \right\},$$}
where 
\begin{align*}
   \delta_\Pi =  \frac{2\epsilon^2_{n, \Pi}}{C\alpha}, \quad\text{and}\quad 
  \delta_\beta = \frac{2\epsilon^2_{n, \beta}}{n},
\end{align*}
and $||\beta_0||_2^2= C$.    
Further, suppose \(\epsilon_n\) satisfies
$$\epsilon^2_n =  \frac{2 (\epsilon_{n, \beta}^2 + \epsilon_{n, \Pi}^2)}{\alpha},$$ such that $\epsilon^2_{n, \Pi} = (\log k + k\log n)/n$, \quad and $\epsilon^2_{n, \beta} = 2/n$.
Then, we have
\begin{equation*}
        \pi_\theta(\mathscr{A}) \geq \exp\left(-n \epsilon_n^2\right),
\end{equation*}
where \(\pi_\theta\) denotes the joint prior distribution on \(\theta = (\Pi, \beta)\).
\end{theorem}

\begin{proof}
For simplicity of exposition, we assume $\sigma^2 = 1$. Then, the $\alpha$-Rényi divergence between \( P_\theta \) and the true data generating mechanism \( P_{\theta_0} \) under the proposed framework is as follows
\begin{align*}
D^{(n)}_{\alpha, \theta_0}(P_\theta, P_{\theta_0}) 
= \frac{\alpha}{2} \left\| \Pi_0 X \beta_0 - \Pi X \beta \right\|_2^2.
\end{align*}
Further, we define the set \(\mathscr{A}'\) as:
{\footnotesize
\begin{align}\label{eqn:proof_1}
\mathscr{A}' = \bigg\{\theta = (\Pi, \beta) : \left\| \Pi_0 X \beta_0 - \Pi X \beta \right\|_2^2 \leq \frac{n (\epsilon_{n, \beta}^2 + \epsilon_{n, \Pi}^2)}{\frac{\alpha}{2}}\bigg\}.
\end{align}}
Suppose we assume \(\|X\|_{2}^2  = n\). One can achieve this via scaling the design matrix $X$ appropriately. Then, one can simplify the expression of $D^{(n)}_{\alpha, \theta_0}(P_\theta, P_{\theta_0})$ in \eqref{eqn:proof_1}, up to the constant $\alpha/2$, as follows
\begin{align}\label{eqn:proof_2}
    &\quad\left\| \Pi_0 X \beta_0 - \Pi X \beta \right\|_2^2 \notag\\
    & = \left\| \Pi_0 X \beta_0 - \Pi X \beta_0\ + \ \Pi X \beta_0 - \Pi X \beta \right\|_2^2 \nonumber \\
    & \leq \left\| \Pi_0 X \beta_0 - \Pi X \beta_0 \right\|_2^2\ +\ \left\|\Pi X \beta_0 - \Pi X \beta \right\|_2^2 \nonumber \\
    & \leq n\ \|\Pi - \Pi_0\|_F^2\ \|\beta_0\|_2^2\ +\ n\ \|\Pi\|_F^2\ \|\beta - \beta_0\|_2^2 \nonumber \\
    & \leq n\ \|\Pi - \Pi_0\|_F^2\ \|\beta_0\|_2^2\  +\ n^2\  \|\beta - \beta_0\|_2^2 \nonumber \\
    & = n\ C\ \|\Pi - \Pi_0\|_F^2\ +\ n^2\ \|\beta - \beta_0\|_2^2.
\end{align}
The first inequality above follows from triangle inequality. Further, note that for any permutation matrix $\Pi = ((\pi_{i, j}))$ or order $n$, we have $||\Pi||^2_{2} = \sum_{i=1}^n\sum_{j=1}^n \pi^2_{i, j} = n$.
Rest of the simplification is complete by noting that $ \|\beta_0\|_2^2 = C$ and \(\|X\|_{2}^2  = n\).

From the simplifications in \eqref{eqn:proof_2}, since $\mathscr{A}\subset\mathscr{A'}$, one can write
\begin{align}\label{pm_main}
    &\quad\pi_\theta(\mathscr{A'}) \geq \pi_\Pi\left( \Pi: \|\Pi - \Pi_0\|_F^2 \leq \delta_\Pi \right)  + \notag\\
    &\quad \pi_\beta\left( \beta: \|\beta - \beta_0\|_2^2 \leq \delta_{\beta} \right) 
    \geq \pi_\theta(\mathscr{A}),
\end{align}
where $\delta_\Pi =  \frac{2\epsilon^2_{n, \Pi}}{C\alpha}$, and
$\delta_\beta = \frac{2\epsilon^2_{n, \beta}}{n}$. Next, we shall look into the two quantities one by one.

First, we consider 
\begin{align}\label{pm_beta}
    &\quad \pi_\Pi(||\Pi -\Pi_0||^2_{\rm F}\leq \delta_{\Pi})
    \geq\ \mbox{\rm P}(\Pi = \Pi_0),\notag\\
    & = \frac{1}{\binom{n}{1} + \binom{n}{2} + \binom{n}{k}}
    \geq\ \frac{1}{k\binom{n}{k}}\geq\ \frac{1}{kn^k},\notag\\
    & = \exp\bigg[-n \bigg(\frac{\log k + k\log n}{n}\bigg)\bigg].
\end{align}
Then, we set $\epsilon^2_{n, \Pi} = (\log k + k\log n)/n$, and we have 
\begin{align}\label{eqn:epsilon_pi}
    n\epsilon^2_{n, \Pi}\geq 2.
\end{align}

Next, we consider   \(\beta \sim N(0, \phi I_p)\) and simplify as follows
\begin{align}\label{pm_pi}
    &\quad \pi_\beta(||\beta -\beta_0||^2_{2}\leq \delta_{\beta})\notag\\
    & \geq \prod_{i=1}^p\pi\bigg(|\beta_{i} -\beta_{0, i}|\leq \sqrt{\frac{\delta_{\beta}}{p}}\bigg)
    = \bigg[\pi\bigg(|\beta_{1} -\beta_{0, 1}|\leq \sqrt{\frac{\delta_{\beta}}{p}}\bigg)\bigg]^p,\notag\\
    & = \bigg[\int_{|\beta_{1} -\beta_{0, 1}|\leq \sqrt{\delta_{\beta}/p}}\frac{1}{2\pi\phi}\exp-\bigg(\frac{\beta^2_1}{2\phi}\bigg)d\beta_1\bigg]^p,\notag\\
    & = \bigg[2M\ \sqrt{\frac{\delta_\beta}{p}}\bigg]^p
     = \exp(-K^{\prime}p\log n),
\end{align}
where $K^{\prime}$ is a constant. The first inequality holds by triangle inequality. The next two steps follow from the assumption $\beta \sim N(0, \phi I_p)$. The penultimate step follows from the mean value theorem \cite{stewart2015calculus}. Simple algebraic manipulations complete the proof. Then, we set $\epsilon^2_{n, \beta} = 2/n$, and we have 
\begin{align}\label{eqn:epsilon_beta}
    n\epsilon^2_{n, \beta}\geq 2.
\end{align}
From \eqref{eqn:epsilon_pi} and \eqref{eqn:epsilon_beta}, we have 
$\frac{n (\epsilon_{n, \beta}^2 + \epsilon_{n, \Pi}^2)}{\frac{\alpha}{2}} \geq 2,$ since $\alpha\in(0, 1)$.
Combining \eqref{pm_main}, \eqref{pm_beta} and \eqref{pm_pi}, we have
\begin{align*}
\pi_\theta(\mathscr{A})
&\geq\pi_\Pi\left(\|\Pi - \Pi_0\|_F^2 \leq \delta_\Pi\right) \pi_\beta\left(\|\beta - \beta_0\|_2^2 \leq \delta_\beta\right) \nonumber \\
&\geq \exp\bigg(-n\frac{\epsilon_{n, \beta}^2 + \epsilon_{n, \Pi}^2}{\frac{\alpha}{2}}\bigg).
\end{align*}
Hence, we have the proof for the prior mass condition, under the proposed framework.
\end{proof}

\begin{corollary}\label{post_contract}
    Assume \( n\epsilon_n^2 \geq 2 \), where $\epsilon_n$ is defined in Theorem \eqref{prior_mass_cond}. Let $\theta = (\Pi, \beta)$, and $\theta_0 = (\Pi_0, \beta_0)$ represents the true parameter values.
    Then, given \( \alpha \in (0, 1) \), for any \( D \geq 2 \) and \( t > 0 \), the following holds
        $\Pi_{n, \alpha}\left(\frac{1}{n}D^{(n)}_{\theta_0, \alpha}(\theta, \theta_0) \geq \frac{D + 3t}{1 - \alpha} \epsilon_n^2 \bigg| W^{(n)}\right) \leq \exp(-tn\epsilon_n^2)$
    with probability at least \( 1 - 2/(D - 1 + t)^2 n \epsilon_n^2 \) under \( \mathbb{P}_{\theta_0}^{(n)} \).
\end{corollary}

Next, we conduct a theoretical analysis of the alternative formulation of the linear regression with sparsely permuted data problem as presented in \eqref{eqn:alt_mod2}, and we establish novel results concerning posterior contraction.

\subsection{Posterior contraction in the alternative formulation in \ref{ssec:alternative}}
Suppose we define \(\theta = (\beta, \sigma^2_{f})\), and \(\theta_0 = (\beta_0, \sigma^2_{f0})\) represents the true parameter values. With that, we state and prove the prior mass condition in Theorem \ref{prior_mass_cond_alt}, which is essential for establishing the posterior contraction rate of the proposed method in \eqref{ssec:alternative}.

\begin{theorem}\label{prior_mass_cond_alt}
Assume $\sigma^2 = 1$. For  $\alpha \in (0,1)$, define the neighbourhood 
{\small
$$\mathscr{A} = \left\{ \theta = (\beta, \sigma^2_f) :\|\beta - \beta_0\|_2^2 \leq \delta_\beta, \bigg|\frac{1 +\sigma^2_{f0}}{1 +\sigma^2_{f}} - 1\bigg| \leq \delta_{\sigma^2_{f}}\right\},$$}
where 
$\delta_{\sigma^2_f} = \delta_{\sigma^2_{f}} =  2(1-\alpha)\epsilon^2_{n, \sigma^2_{f0}}$ and
  $\delta_\beta = \frac{2\epsilon^2_{n, \beta}}{n\alpha}$.
Further, suppose \(\epsilon_n\) satisfies
$\epsilon^2_n = \bigg[\frac{\epsilon_{n, \beta}^2}{\frac{\alpha}{2}} + \frac{\epsilon_{n, \sigma^2_f}^2}{\frac{1}{2(1-\alpha)}}\bigg]$,
such that $\epsilon^2_{n, \beta} = 2/n$, and $\epsilon^2_{n, \sigma^2_f} = 2/n$.
Then, we have $\pi_\theta(\mathscr{A}) \geq \exp\left(-n \epsilon_n^2\right)$,
where \(\pi_\theta\) denotes the joint prior distribution on \(\theta = (\beta, \sigma^2_f)\).
\end{theorem}
\begin{proof}
For simplicity of exposition, we assume $\sigma^2 = 1$.  Suppose we assume \(\|X\|_{2}^2  = n\).  Further, we define the set \(\mathscr{A}'\) as
{\footnotesize
\begin{align}\label{eqn:alt_proof_1}
\mathscr{A}' = \bigg\{\theta = (\beta, \sigma^2_f) : D^{(n)}_{\alpha, \theta_0}(P_\theta,\ P_{\theta_0})\leq n\epsilon^2_n = n\bigg[\frac{\epsilon_{n, \beta}^2}{\frac{\alpha}{2}} + \frac{\epsilon_{n, \sigma^2_f}^2}{\frac{1}{2(1-\alpha)}}\bigg]\bigg\}.
\end{align}}
Under the proposed framework in \eqref{ssec:alternative}, the $\alpha$-Rényi divergence between $$P_\theta\equiv \bigg(1-\frac{k}{n}\bigg) \mbox{N}(X\beta, I) +\frac{k}{n} \mbox{N}(X\beta, (1 + \sigma^2_f) I),$$ and the true data generating mechanism $$ P_{\theta_0}\equiv \bigg(1-\frac{k}{n}\bigg) \mbox{N}(X\beta_0, I) +\frac{k}{n} \mbox{N}(X\beta_0, (1 + \sigma^2_{f0}) I), $$ can be bounded as follows
{\small
\begin{align}\label{eqn:alt_proof_2}
&\quad D^{(n)}_{\alpha, \theta_0}(P_\theta,\ P_{\theta_0}) \notag\\
&\leq \bigg(1-\frac{k}{n}\bigg) D^{(n)}_{\alpha, \theta_0}(\mbox{N}(X\beta, I), \mbox{N}(X\beta_0, I))\ +\  \notag\\
&\quad \ \frac{k}{n}\ D^{(n)}_{\alpha, \theta_0}(\mbox{N}(X\beta, (1 + \sigma^2_{f}),\ \mbox{N}(X\beta_0, (1 + \sigma^2_{f0}) ),\notag\\
&=\frac{\alpha}{2}\bigg[\bigg(1-\frac{k}{n}\bigg) + \frac{k/n}{\alpha(1+\sigma^2_{f0})+(1-\alpha)(1+\sigma^2_{f}) }\bigg]\notag\\
&\quad\ \ ||X\beta - X\beta_0||^2_2 + \frac{k}{n}\frac{1}{2(1-\alpha)}\frac{[\alpha(1+\sigma^2_{f0}) + (1-\alpha)(1+\sigma^2_{f})]^p}{[(1+\sigma^2_{f})^{p(1-\alpha)}]\times[(1+\sigma^2_{f0})^{p\alpha}]},\notag\\
&\leq \frac{n\alpha}{2}||\beta - \beta_0||^2_2\ + \ \frac{k}{n}\frac{1}{2(1-\alpha)}\frac{\bigg[\alpha\frac{(1+\sigma^2_{f0})}{(1+\sigma^2_{f})} + (1-\alpha)\bigg]^p}{\bigg[\frac{(1+\sigma^2_{f0})}{(1+\sigma^2_{f})}\bigg]^{\alpha p}},\notag\\
&\leq \frac{n\alpha}{2}||\beta - \beta_0||^2_2\ + \ \frac{k}{n}\frac{1}{2(1-\alpha)}\bigg[\frac{1-\alpha}{1 -\delta_{\sigma^2_f}}\bigg]^p\notag\\
&\leq \frac{n\alpha}{2}||\beta - \beta_0||^2_2\ + \ \frac{k}{n}\frac{1}{2(1-\alpha)},
\end{align}}
provided $\alpha > \delta_{\sigma^2_f}$. The first inequality follows from the convexity of $ D^{(n)}_{\alpha, \theta_0}(\cdot,\cdot)$ in both the arguments. The second step utilizes the expression for Renyi divergence between two multivariate normal distributions \cite{GIL2013124}. Rest are simple algebraic manipulation.

From the simplifications in \eqref{eqn:alt_proof_2}, since $\mathscr{A}\subset\mathscr{A'}$, one can write
\begin{align*}
&\pi_\theta(\mathscr{A'})
\geq  \pi_\beta\left( \beta: \|\beta - \beta_0\|_2^2 \leq \delta_{\beta} \right) + \notag\\
&\pi_{\sigma^2_f}\bigg(\sigma^2_{f}: \bigg|\frac{1 +\sigma^2_{f0}}{1 +\sigma^2_{f}} - 1\bigg| \leq \delta_{\sigma^2_{f}}\bigg) 
\geq \pi_\theta(\mathscr{A}),
\end{align*}
where $\delta_{\sigma^2_{f}} =  \epsilon^2_{n, \sigma^2_{f0}}$, and
$\delta_\beta = \frac{2\epsilon^2_{n, \beta}}{n\alpha}$. 

First, we consider   \(\beta \sim \mbox{N}(0, \phi I_p)\) and simplify as earlier, 
{\small
\begin{align*}
\pi_\beta(||\beta -\beta_0||^2_{2}\leq \delta_{\beta})
\geq \bigg[2M\ \sqrt{\frac{\delta_\beta}{p}}\bigg]^p
     = \exp(-K^{\prime}p\log n),  
\end{align*}}
where $K^{\prime}$ is a constant. Then, we set $\epsilon^2_{n, \beta} = 2/n$, and we have $n\epsilon^2_{n, \beta}\geq 2$. 
Next, we consider  an usual prior on \(\sigma^{-2}_f \sim \mbox{\rm Gamma}(\kappa_1, \kappa_2)\). We set $\kappa_1=1$ for simplicity of exposition. Then, 
{\small
\begin{align*}
&\quad\pi_{\sigma^2_f}\bigg(\bigg|\frac{1+\sigma^2_{f0}}{1+\sigma^2_{f}} - 1\bigg|\leq \delta_{\sigma^2_f}\bigg)\\
&=\int_{\frac{1-\delta_{\sigma^2_f}}{\sigma^2_{f0} +\delta_{\sigma^2_f}}}^\frac{1+\delta_{\sigma^2_f}}{\sigma^2_{f0} -\delta_{\sigma^2_f}}\kappa_2[e^{-\kappa_2 u}] du\geq \int_{\frac{1-\delta_{\sigma^2_f}}{\sigma^2_{f0} -\delta_{\sigma^2_f}}}^\frac{1+\delta_{\sigma^2_f}}{\sigma^2_{f0} -\delta_{\sigma^2_f}}\kappa_2[e^{-\kappa_2 u}] du\notag\\
&\geq\ 2\kappa_2 M^{\prime}\ \bigg(\frac{\delta_{\sigma^2_f}}{\sigma^2_{f0} -\delta_{\sigma^2_f}}\bigg)\ 
\geq\ 2\kappa_2 M^{\prime}\ \bigg(\frac{\delta_{\sigma^2_f}}{\sigma^2_{f0}}\bigg)\notag\\
&=\ K^{\prime\prime}\ \delta_{\sigma^2_f} \ =\ K^{\prime\prime}\exp (-\log n),
\end{align*}}
where $K^{\prime\prime}>0$ is a constant. Then, we set $\epsilon^2_{n, \delta_{\sigma^2_f}} = 2/n$, and we have $n\epsilon^2_{n, \delta_{\sigma^2_f}}\geq 2$. Next, we set 
$n\epsilon^2_n = n\bigg[\frac{\epsilon_{n, \beta}^2}{\frac{\alpha}{2}} + \frac{\epsilon_{n, \sigma^2_f}^2}{\frac{1}{2(1-\alpha)}}\bigg] \geq 2,$ since $\alpha\in(0, 1)$, and  we have the proof for the prior mass condition, under the alternative formulation in \eqref{eqn:alt_mod2}.
\end{proof}
\begin{corollary}\label{post_contract}
    Assume \( n\epsilon_n^2 \geq 2 \), where $\epsilon_n$ is defined in Theorem \eqref{prior_mass_cond_alt}. Let $\theta = (\beta,\sigma^2_f)$, and $\theta_0 = (\beta_0, \sigma^2_{f0})$ represents the true parameter values.
    Then, given \( \alpha \in (0, 1) \), for any \( D \geq 2 \) and \( t > 0 \), the following holds
        $\Pi_{n, \alpha}\left(\frac{1}{n}D^{(n)}_{\theta_0, \alpha}(\theta, \theta_0) \geq \frac{D + 3t}{1 - \alpha} \epsilon_n^2 \bigg| W^{(n)}\right) \leq \exp(-tn\epsilon_n^2)$
    with probability at least \( 1 - 2/(D - 1 + t)^2 n \epsilon_n^2 \) under \( \mathbb{P}_{\theta_0}^{(n)} \).
\end{corollary}

\section{Performance evaluation}

\subsection{Linear regression}

We first generate $n$ observations from the linear regression model $y_i = \beta^{T} \mathbf{x}_i + \epsilon_i,\ i\in [n]$, where $\epsilon_i\stackrel{i.i.d}{\sim}\mbox{N}(0, \sigma^2)$. Then, we first permute the first $s_0$ observations via multiplying the design matrix $X= [\mathbf{x}_1,\ldots, \mathbf{x}_n]$ by 
\[
\begin{pmatrix}
\begin{array}{ccccc|cccc}
0 & 0 & \cdots & 0& 1 & 0 & 0& \cdots & 0 \\
0 & 0 & \cdots &1&0 & 0 & 0& \cdots & 0 \\
\vdots & \vdots & \vdots & \vdots &\vdots & \vdots & \vdots & \vdots&\vdots \\
1 & 0 & \cdots & 0 & 0 &0& \cdots & 0 &0\\
\hline
0 & 0 & \cdots & 0 & 0& 1 & 0& \cdots & 0 \\
0 & 0 & \cdots & 0 & 0 &0 &  1&\cdots & 0 \\
\vdots & \vdots & \vdots & \vdots & \vdots & \vdots &\vdots & \vdots&\vdots \\
0 & 0 & \cdots & 0 & 0 & 0& 0& \cdots & 1 \\
\end{array}
\end{pmatrix},
\]
and then repeat the observation $(\mathbf{x}_1, y_1)$ once as the $(s_0+1)$-th observation, so that the posterior mean of $\Pi_0$ is
\[
\begin{pmatrix}
\begin{array}{cccccc|cccc}
0 & 0 & \cdots &0& 0& 1 & 0 & 0& \cdots & 0 \\
0 & 0 & \cdots&0 &1&0 & 0 & 0& \cdots & 0 \\
0 & 0 & \cdots &1&0&0 & 0 & 0& \cdots & 0 \\
\vdots & \vdots & \vdots & \vdots &\vdots & \vdots & \vdots & \vdots&\vdots \\
0.5 & 0.5 & \cdots & 0 & 0 &0& \cdots & 0 &0&0\\
0.5 & 0.5 & \cdots & 0 & 0 &0& \cdots & 0 &0&0\\
\hline
0 & 0 & \cdots & 0 & 0& 0& 1 & 0& \cdots & 0 \\
0 & 0 & \cdots & 0 & 0 &0 & 0& 1&\cdots & 0 \\
\vdots & \vdots & \vdots & \vdots& \vdots & \vdots & \vdots &\vdots & \vdots&\vdots \\
0 & 0 & \cdots & 0 & 0 & 0& 0&0& \cdots & 1 \\
\end{array}
\end{pmatrix}.
\]
We set $s_0= 6$, $\sigma = 0.1$, and $\beta = (1,\ldots, 1)^T\in\mathbf{R}^{20}$.
Our goal is to conduct full Bayesian inference of the parameters $\Pi$ and $(\beta,\sigma^2)$ via the proposed methodology in \ref{ssec:main_model}, for varying sample size $n\in\{100, 150, 200, 250\}$ and values of the temperature parameter $\kappa\in(0, 1)$.

In Figure \ref{fig:reg_coef}, the posterior distributions of \(\beta\) for varying \(\kappa \in \{1/n, 0.99\}\) with a sample size of \(n=100\) appear largely similar. However, as shown in Figure \ref{fig:pi_posterior}, the proposed methodology with \(\kappa=1/n\) achieves superior recovery of the permutation matrix \(\Pi\) compared to the case where \(\kappa=0.99\) is used. This highlights the effectiveness of employing an adaptive temperature parameter \(\kappa \approx 1/n\) in enhancing the recovery of the parameters of interest.

We consolidate the aforementioned findings through extensive repeated simulations, varying the sample size \(n \in \{100, 150, 200, 250\}\) and the temperature parameter \(\kappa \in \{1/n, 0.1, 0.5, 0.75, 0.99\}\). The results, summarized in Tables \ref{tab:performance1} and \ref{tab:performance2}, indicate that the recovery of both \(\beta\) and \(\Pi\) improves when a lower temperature parameter \(\kappa\) is utilized.

On the computational front, the MCMC algorithms proposed in Section \ref{ssec:mcmc} are meticulously designed to provide a unified computational scheme for any \(\kappa \in (0, 1]\). As demonstrated in Figure \ref{fig:linreg_timing}, the computational times remain comparable across varying \(\kappa\) for a fixed sample size \(n\). Additionally, the proposed computational scheme scales favorably with increasing \(n\).

\begin{figure}[htbp]
    \centering
    \begin{subfigure}[b]{0.45\textwidth}
        \centering
        \includegraphics[width=\textwidth,height=0.2\textheight]{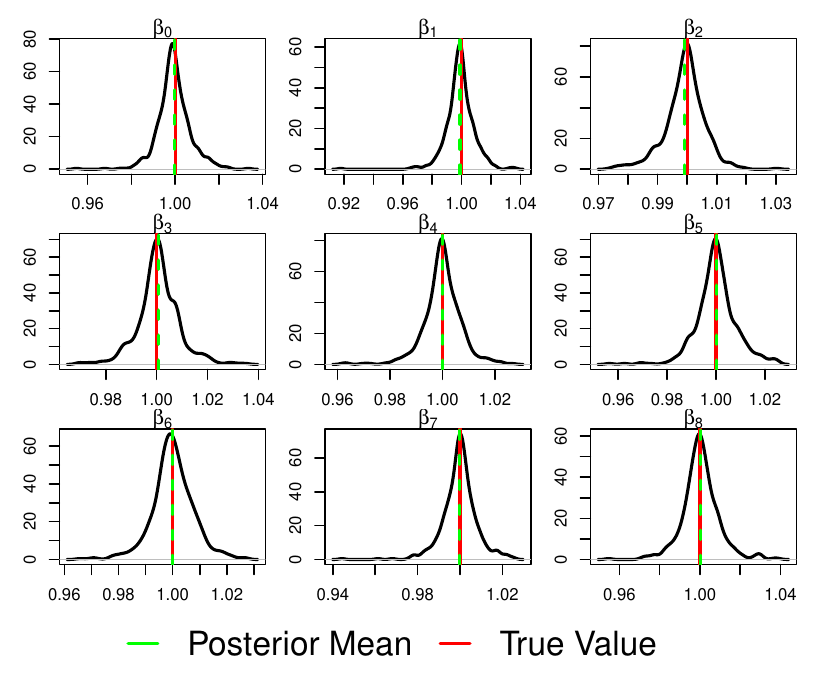}
        \label{fig:alpha001}
    \end{subfigure}
    \hfill
    \begin{subfigure}[b]{0.45\textwidth}
        \centering
        \includegraphics[width=\textwidth,height=0.2\textheight]{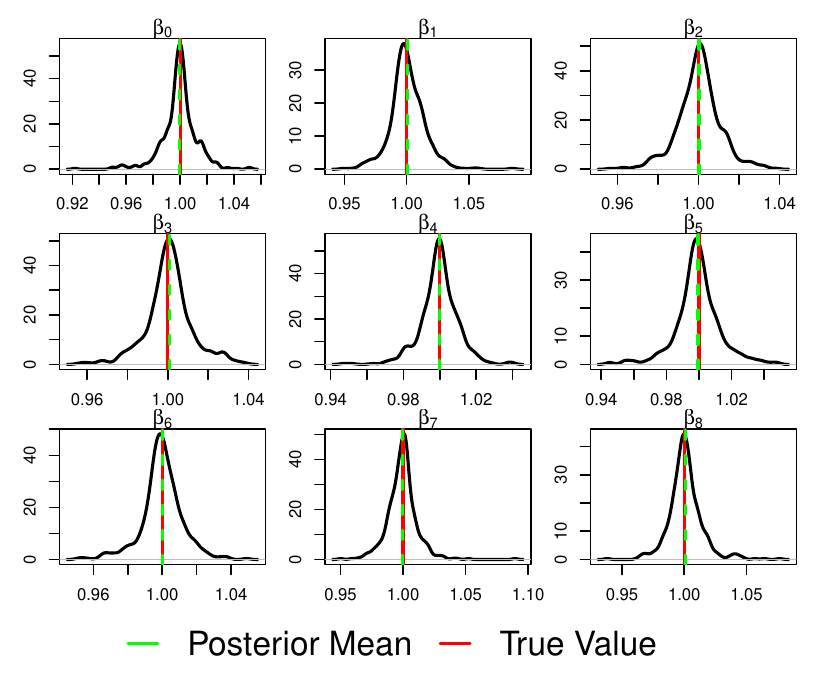}
        \label{fig:alpha99}
    \end{subfigure}
    \caption{\textbf{Linear regression.} Posterior distributions of $\beta$ for $\alpha=\frac{1}{n}$ at the \textbf{top}, for $\alpha=0.99$ at the \textbf{bottom}.}
    \label{fig:reg_coef}
\end{figure}

\begin{figure}[htbp]
    \centering
    \begin{subfigure}[b]{0.45\textwidth}
    \includegraphics[width=\textwidth,height=0.2\textheight]{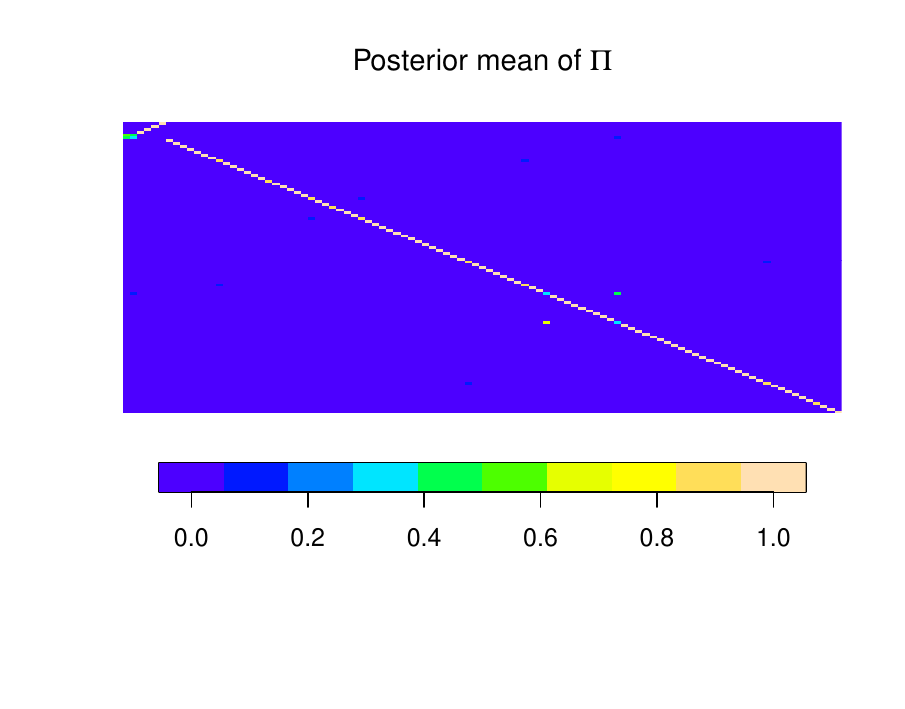}
    \end{subfigure}
    \begin{subfigure}[b]{0.45\textwidth}
    \includegraphics[width=\textwidth,height=0.2\textheight]{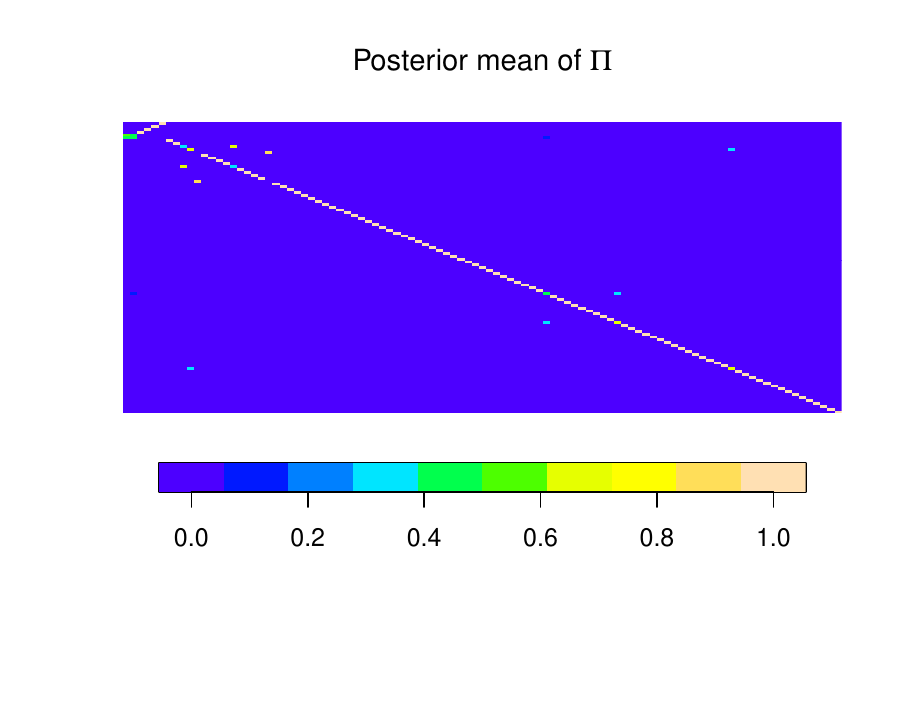}
    \end{subfigure}
    \caption{\textbf{Linear regression.} Posterior mean of $\Pi$ for $\alpha=\frac{1}{n}$ at the \textbf{top}, for $\alpha=0.99$ at the \textbf{bottom}. 
    }
    \label{fig:pi_posterior}
\end{figure}


\begin{table}[!htb]
\caption{\textbf{Linear regression.} The maximum $\mathcal{L}_1$ error (in $10^{-4}$ unit) for varying temperature $\kappa\in\{\frac{1}{n}, 0.1, 0.5, 0.75, 0.99\}$ and sample size $n\in\{100, 150\}$, calculated via repeated simulations. Here $\beta_0$ and $\Pi_0$ denote the true values of the parameters $(\beta, \Pi)$, and $\hat{\beta}$ and $\hat{\Pi}$ denote the corresponding posterior mean. 
}
\label{tab:performance1}
\small\setlength\tabcolsep{4pt}
\scalebox{0.9}{
\begin{tabular}{|c|c|c|c|c|}
    \hline
      \multicolumn{1}{|c|}{} &
      \multicolumn{2}{|c|}{$n= 100$} &
      \multicolumn{2}{|c|}{$n= 150$}\\
    \hline
     \multicolumn{1}{|c|}{$\kappa$ } &
      \multicolumn{1}{|c|}{$\frac{1}{p}\sum_{i=1}^p |\hat{\beta}_i - \beta_i|$} &
      \multicolumn{1}{|c|}{$d_H(\hat{\Pi}, \Pi_0)$} &
      \multicolumn{1}{|c|}{$\frac{1}{p}\sum_{i=1}^p |\hat{\beta}_i - \beta_i|$} &
      \multicolumn{1}{|c|}{$d_H(\hat{\Pi}, \Pi_0)$} \\
      \hline
      $\frac{1}{n}$ & $9.63$ & $4.19$& $8.90$ & $2.36$\\
\hline
$0.1$ &$12.20$ & $8.63$ & $9.86$ & $5.07$\\
\hline
$0.5$ & $14.58$& $7.25$ & $13.32$ & $5.70$\\
\hline
$0.75$ & $20.67$ & $11.06$& $12.65$ & $3.44$ \\
\hline
$0.99$ & $19.74$ & $10.54$& $17.25$& $6.48$\\
\hline
\end{tabular}
}
\end{table}

\begin{table}[!htb]
\caption{\textbf{Linear regression.} The maximum $\mathcal{L}_1$ error (in $10^{-4}$ unit) for varying temperature $\kappa\in\{\frac{1}{n}, 0.1, 0.5, 0.75, 0.99\}$ and sample size $n\in\{200, 250\}$. Here $\beta_0$ and $\Pi_0$ denote the true values of the parameters $(\beta, \Pi)$, and $\hat{\beta}$ and $\hat{\Pi}$ denote the corresponding posterior mean. 
}
\label{tab:performance2}
\small\setlength\tabcolsep{4pt}
\scalebox{0.9}{
\begin{tabular}{|c|c|c|c|c|}
    \hline
      \multicolumn{1}{|c|}{} &
      \multicolumn{2}{|c|}{$n= 200$} &
      \multicolumn{2}{|c|}{$n= 250$}\\
    \hline
     \multicolumn{1}{|c|}{$\kappa$ } &
      \multicolumn{1}{|c|}{$\frac{1}{p}\sum_{i=1}^p |\hat{\beta}_i - \beta_i|$} &
      \multicolumn{1}{|c|}{$d_H(\hat{\Pi}, \Pi_0)$} &
      \multicolumn{1}{|c|}{$\frac{1}{p}\sum_{i=1}^p |\hat{\beta}_i - \beta_i|$} &
      \multicolumn{1}{|c|}{$d_H(\hat{\Pi}, \Pi_0)$} \\
\hline
$\frac{1}{n}$ & $7.88$ & $0.89$& $8.59$ & $0.63$\\
\hline
$0.1$ & $7.84$& $ 0.83$& $8.33$& $ 0.97$\\
\hline
$0.5$ &$9.02$ &$3.42$ & $11.79$ & $0.90$\\
\hline
$0.75$ &$13.21$ & $3.68$& $12.52$& $1.16$\\
\hline
$0.99$ & $13.53$& $3.31$& $18.57$ & $1.24$\\
\hline
\end{tabular}
}
\end{table}

\begin{figure}[htbp]
    \begin{subfigure}[b]{0.5\textwidth}
        \centering
        \includegraphics[width=\textwidth,height=0.2\textheight]{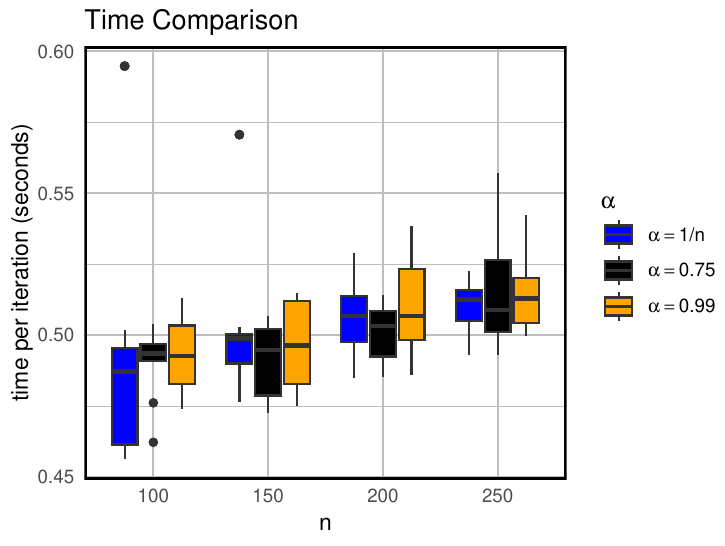}
    \end{subfigure}
    \caption{\textbf{Linear regression.} The per iteration time (in seconds) for varying temperature $\kappa\in\{\frac{1}{n}, 0.75, 0.99\}$ and sample size $n\in\{100, 150, 200, 250\}$. }
    \label{fig:linreg_timing}
\end{figure}
\subsection{Quantile regression}
To illustrate the flexibility of the proposed methodology beyond linear regression models, we extend it to quantile regression with sparsely permuted data. Further, we present findings from small-scale numerical experiments to demonstrate its efficacy in this context.

Suppose we observe data \((y_i, \mathbf{x}_i) \in \mathbb{R} \times \mathbb{R}^d\) for \(i \in [n]\), with mismatches occurring in at most \(k\) entries, where \(k \ll n\). We model the data as follows
\begin{align}\label{eqn:model1_quant}
 \mathbf{y} = \Pi X \beta + \boldsymbol{\epsilon}, \quad \epsilon \stackrel{i.i.d.}{\sim} \text{ALD}(\mathbf{0}, \sigma, \tau),
\end{align}
where \(\text{ALD}(\mathbf{0}, \sigma^2, \tau)\) denotes an asymmetric Laplace distribution with probability density function given by
$
    f_{\epsilon, \tau}(u \mid \sigma) = \frac{1}{\sigma} \exp \left\{ -\frac{\rho_\tau(u)}{\sigma} \right\},\ u\in\mathbf{R},
$
where \(\rho_\tau(u) = u (\tau - \mathbf{1}_{\{u < 0\}})\) is the check function, and \(\tau \in [0, 1]\) is the quantile level. The permutation matrix \(\Pi = ((\pi_{ij}))_{i,j=1}^n\) is assumed to satisfy \(d_H(\Pi, I_n) = \sum_{i=1}^n \sum_{j=1}^n |\pi_{ij} - 1| \leq k\), where \(d_H\) denotes the Hamming distance and \(I_n\) is the identity matrix of size \(n\). Our objective is to develop a fully Bayesian framework to estimate the parameters \(\Pi\) and \((\beta, \sigma^2)\). We begin by specifying prior distributions for \((\Pi, \beta, \sigma^2)\) as detailed in Section \ref{ssec:main_model}. We then adapt the computational scheme described in Section \ref{ssec:mcmc} to accommodate the new model. The modifications to the computation are straightforward; therefore, detailed explanations are omitted here due to space constraints.

\begin{table}[!htb]
\caption{\textbf{Quantile regression.} The maximum $\mathcal{L}_1$ error (in $10^{-4}$ unit) for varying temperature $\kappa\in\{\frac{1}{n}, 0.99\}$ and sample size $n\in\{100, 150,  250\}$, calculated via repeated simulations. Here $\beta_0$ and $\Pi_0$ denote the true values of the parameters $(\beta, \Pi)$, and $\hat{\beta}$ and $\hat{\Pi}$ denote the corresponding posterior mean. 
}
\label{tab:performance3}
\small\setlength\tabcolsep{4pt}
\scalebox{0.9}{
\begin{tabular}{|c|c|c|c|c|}
    \hline
      \multicolumn{1}{|c|}{} &
      \multicolumn{2}{|c|}{$\kappa= 0.99$} &
      \multicolumn{2}{|c|}{$\kappa= 1/n$}\\
    \hline
     \multicolumn{1}{|c|}{$n$ } &
      \multicolumn{1}{|c|}{$\frac{1}{p}\sum_{i=1}^p |\hat{\beta}_i - \beta_i|$} &
      \multicolumn{1}{|c|}{$d_H(\hat{\Pi}, \Pi_0)$} &
      \multicolumn{1}{|c|}{$\frac{1}{p}\sum_{i=1}^p |\hat{\beta}_i - \beta_i|$} &
      \multicolumn{1}{|c|}{$d_H(\hat{\Pi}, \Pi_0)$} \\
      \hline
      \hline
100 & 4.46 & 3.85 & 3.70 & 1.73 \\
\hline
150& 3.66 & 1.76 & 2.93 & 1.39 \\
\hline
\hline
250 &3.11 & 0.25 & 2.38 & 0.23 \\
\hline
\end{tabular}
}
\end{table}

\begin{figure}[htbp]
    \begin{subfigure}[b]{0.5\textwidth}
        \centering
        \includegraphics[width=\textwidth,height=0.2\textheight]{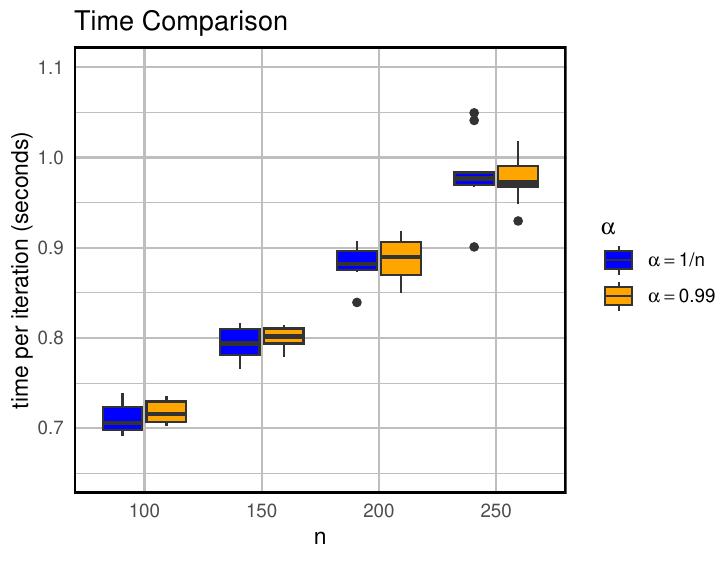}
    \end{subfigure}
    \caption{\textbf{Quantile regression.} The per iteration time (in seconds) for varying temperature $\kappa\in\{\frac{1}{n}, 0.99\}$ and sample size $n\in\{100, 150, 200, 250\}$. }
    \label{fig:pi_posterior}
\end{figure}

We report findings from a  small-scale repeated simulation, where we varied the sample size \(n \in \{100, 150,  250\}\) and the temperature parameter \(\kappa \in \{1/n, 0.99\}\). The results, as summarized in Table \ref{tab:performance2}, show that the recovery of both \(\beta\) and \(\Pi\) improves with the use of a lower temperature parameter \(\kappa\). Furthermore, Figure \ref{fig:linreg_timing} illustrates that computational times remain comparable across different values of \(\kappa\) for a fixed sample size \(n\). Additionally, the proposed computational scheme scales effectively with increasing \(n\).

\section*{Acknowledgment}
There was no internal or external funding for this work.

\bibliographystyle{ieeetr}
\bibliography{paper-ref, paper-ref_dbetel}
\vspace{12pt}
\color{red}

\end{document}